\documentclass[reqno,12pt]{amsart}
\usepackage{a4wide}

\usepackage{amsmath}
\usepackage{amsfonts}
\usepackage{amssymb}
\usepackage{amscd}
\usepackage{stmaryrd}
\usepackage[matrix,arrow,cmtip,rotate,curve,arc]{xy}
\SelectTips{cm}{10}

\newcommand{\bbZ}{\mathbb Z}
\newcommand{\bbR}{\mathbb R}

\newcommand{\mcB}{\mathcal B}
\newcommand{\mcC}{\mathcal C}
\newcommand{\mcG}{\mathcal G}
\newcommand{\mcR}{\mathcal R}

\newcommand{\Ra}{\Rightarrow}
\newcommand{\La}{\Leftarrow}
\newcommand{\LRa}{\Leftrightarrow}

\newcommand{\xra}[1]{\xrightarrow{#1}}

\newcommand{\ra}{\rightarrow}

\newcommand{\lra}{\longrightarrow}

\newcommand{\dra}{\xymatrix@1{\ar@{-->}[r]&}}

\newdir{c}{{}*!/-5pt/@^{(}}
\newdir{d}{{}*!/-5pt/@_{(}}
\newdir{ >}{{}*!/-5pt/@{>}}
\newdir{s}{{}*!/+10pt/@{}}
\newdir{|>}{{}*!/2pt/@{|}*@{>}}

\newcommand{\degree}[3]{|#1|_{#2\ldots #3}}
\newcommand{\Hom}{\mathop\mathrm{Hom}\nolimits}
\newcommand{\id}{\mathop\mathrm{id}\nolimits}
\newcommand{\im}{\mathop\mathrm{im}\nolimits}
\newcommand{\Sh}{\mathop\mathrm{Sh}\nolimits}

\newcommand{\sign}[1]{(-1)^{#1}\cdot}

\entrymodifiers={+!!<0pt,\the\fontdimen22\textfont2>}
\theoremstyle{plain}
\newtheorem{theorem}{Theorem}
\newtheorem{lemma}[theorem]{Lemma}
\newtheorem{corollary}[theorem]{Corollary}
\newtheorem{proposition}[theorem]{Proposition}

\theoremstyle{definition}
\newtheorem{definition}[theorem]{Definition}

\newtheorem{example}[theorem]{Example}

\theoremstyle{remark}
\newtheorem*{remark}{Remark}

\newtheoremstyle{tiny}% name
     {10pt}%      Space above
     {10pt}%      Space below
     {\footnotesize}%         Body font
     {}%         Indent amount (empty = no indent, \parindent = para indent)
     {\itshape}% Thm head font
     {.}%        Punctuation after thm head
     {.5em}%     Space after thm head: " " = normal interword space;
     {}%         Thm head spec (can be left empty, meaning `normal')

\theoremstyle{tiny}
\newtheorem*{tremark}{Remark}

\usepackage{enumitem}

\begin{document}

\author{Luk\'a\v s Vok\v r\'inek}

\address{Department of Mathematics and Statistics\\Masaryk University\\Kotl\'a\v rsk\'a 2\\611 37 Brno\\Czech Republic}

\email{koren@math.muni.cz}

\title[Constructing homotopy equivalences of $\bbZ G$-complexes]{Constructing homotopy equivalences of chain complexes of free $\bbZ G$-modules}



\date{December 31, 2012}

\thanks{
The research was supported by the Grant agency of the Czech republic under the grant P201/12/G028.}

\begin{abstract}
We describe a general method for algorithmic construction of $G$-equivariant chain homotopy equivalences from non-equivariant ones. As a consequence, we obtain an algorithm for computing equivariant (co)homology of Eilenberg-MacLane spaces $K(\pi,n)$, where $\pi$ is a finitely generated $\bbZ G$-module.

This result will be used in a forthcoming paper to construct equivariant Postnikov towers of simply connected spaces with free actions of a finite group $G$ and further to compute stable equivariant homotopy classes of maps between such spaces.

The methods of this paper work for modules over any non-negatively graded differential graded algebra, whose underlying graded abelian group is free with $1$ as one of the generators.
\end{abstract}

\maketitle

\section{Introduction}

\subsection*{Notation}
In this paper, $G$ will stand for a fixed finite group. Its (integral) group algebra will be denoted $\mcG=\bbZ G$. A chain complex of free abelian groups will be called a \emph{$\bbZ$-complex}. It is said to be \emph{locally finite}, if it consists of finitely generated abelian groups. Similarly, a chain complex of free $\mcG$-modules will be called a \emph{$\mcG$-complex}. A homomorphism of $\mcG$-modules or chain complexes will be also called a $\mcG$-linear map or an equivariant map.

\subsection*{Introduction}
It is well-known, that a $\bbZ$-complex is homotopy equivalent to a locally finite $\bbZ$-complex if and only if its homology groups are finitely generated. The same is true for $\mcG$-complexes equivariantly.

In this paper, we are interested in constructing such homotopy equivalences algorithmically. Namely, we present an alogrithm that, given a non-equivariant homotopy equivalence of a $\mcG$-complex $M$ with some locally finite $\bbZ$-complex, constructs a $\mcG$-linear homotopy equivalence of $M$ with a locally finite $\mcG$-complex.

Our main application is to the so-called effective algebraic topology, which studies simplicial sets (often infinite) from the effective, i.e.~algorithmic, point of view, e.g.~computes their homotopy groups. These simplicial sets are accessed via their chain complexes which, although infinite, often admit a computable homotopy equivalence with a locally finite chain complex. Introduced by Sergeraert these are called \emph{simplicial sets with effective homology} (see e.g.~\cite{Sergeraert} or \cite{polypost}). Building on this notion we proved in \cite{polypost} that all finite simply connected simplicial sets have a Postnikov tower consisting of simplicial sets with effective homology and indeed with polynomial-time homology (the running times of the algorithms involved are polynomial).

With the results of this paper, we will show in \cite{aslep}, that the same is possible for finite simply connected simplicial sets equipped with a free action of $G$. In the case $G=\bbZ/2$, this leads to the solution of the problem ``Does a simplicial complex $K$ embed into $\bbR^d$?'' in the so-called meta-stable range $\dim K\leq\tfrac{2}{3}d-1$, which was left open in~\cite{embedding}.

In this paper we will show, as a demonstration of our main theorem, that the Eilenberg-MacLane space $K(\pi,n)$, where $\pi$ is a finitely generated $\mcG$-module, can be equipped (equivariantly) with effective homology; this is at the same time an important step in the general case mentioned above. As a consequence, the equivariant homology and cohomology of $K(\pi,n)$ is algorithmically computable. This generalizes the well-known result \cite{EilenbergMacLane} of Eilenberg and MacLane to the equivariant situation.

Our methods are completely general and as such work for modules over an arbitrary differential graded algebra $\mcR$, which is non-negatively graded and whose underlying graded abelian group is free with $1$ as one of the generators. We will thus present most of our results for $\mcR$ and for the applications we will restrict to $\mcG$.

\subsection*{Outline}
The construction of the $\mcG$-linear homotopy equivalence proceeds as follows. Starting from a non-equivariant homotopy equivalence $M\simeq N$ of a $\mcG$-complex $M$ with a $\bbZ$-complex $N$, we put on $N$ the structure of an ``up to homotopy'' $\mcG$-complex. This structure is analogous to the $A_\infty$-structure living on a chain complex homotopy equivalent to a dga --- the $\mcG$-complexes are dg-modules for a dga $\mcG$ concentrated in dimension $0$ and similarly the up to homotopy version is governed by a dga which we call $\mcG_\infty$ (not concentrated in dimension $0$ anymore). The chain maps $M\to N$ and $N\to M$ are not $\mcG_\infty$-linear however. They are some relaxed versions, which we call $\mcG_\infty$-chain maps. The homotopies are not even $\mcG_\infty$-maps!

There is a way of strictifying $\mcG_\infty$-maps by passing to certain ``cofibrant replacements'' $BM$ and $BN$ of the $\mcG_\infty$-modules $M$ and $N$ (it is a perturbation of the usual bar construction) while the new chain homotopies for these strictifications have to be constructed essentially from scratch. In the end we replace the homotopy equivalence $M\simeq N$ by a $\mcG$-linear homotopy equivalence $BM\simeq BN$. The last step is to construct a $\mcG$-linear homotopy equivalence $BM\simeq M$.

We give precise statements in Section~\ref{s:statement} and proofs in Sections~\ref{s:proof_theorem}~and~\ref{s:proof_corollary}.

\section{Basic conventions}

All our chain complexes will be non-negatively graded and of the homological type, i.e.~the differential $\partial$ will be of degree $-1$. This applies also to differential graded algebras where we assume of course that $\partial$ is a graded derivation, $\partial(x\cdot y)=\partial x\cdot y+\sign{|x|}x\cdot\partial y$.

Given two chain complexes $C$ and $D$, we form their tensor product $C\otimes D$ whose degree $n$ part is $(C\otimes D)_n=\bigoplus_{p+q=n}C_p\otimes D_q$. If $f$ and $g$ are two maps, their tensor product $f\otimes g$ is defined by $(f\otimes g)(x\otimes y)=\sign{|g|\cdot|x|}fx\otimes gy$. The differential on $C\otimes D$ is then given as $\partial_\otimes=\partial\otimes\id+\id\otimes\partial$.

For two chain complexes $C$ and $D$ we form the unbounded\footnote{Alternatively one may take its truncation by throwing away all negatively graded pieces and replacing the $0$-chains by $0$-cycles, i.e.~chain maps.} chain complex $\Hom(C,D)$ with
\[\Hom(C,D)_k=\prod_{n=0}^\infty\Hom(C_n,D_{n+k}).\]
and the differential which we denote $[\partial,f]=\partial f-\sign{|f|}f\partial$ (it is a graded commutator) where $|f|$ is the degree of $f$, i.e.~$f\in\Hom(C,D)_{|f|}$. In this way $f$ is a chain map if $[\partial,f]=0$ and $h$ is a homotopy between $f$ and $g$ if $[\partial,h]=g-f$. We will also use the graded Leibniz rule\footnote{There is a version for the tensor product: $[\partial,f\otimes g]=[\partial,f]\otimes g+\sign{|f|}f\otimes[\partial,g]$. The version for the composition is then equivalent to the composition map
\[-\circ-\colon \Hom(D,E)\otimes\Hom(C,D)\ra\Hom(C,E)\]
being a chain map, i.e.~its differential being zero.}
\[[\partial,fg]=[\partial,f]g+\sign{|f|}f[\partial,g].\]
We will denote the suspension of a chain complex $C$ by $sC$. It is defined as $(sC)_n=C_{n-1}$ with differential $-\partial$. A good explaination of this sign change starts by considering the (identity) map $s\colon C\to sC$ and writing elements of $sC$ as $sx$. Postulating $s$ to be a chain map of degree $1$, the differential on $sC$ is forced to
\[\partial sx=-s\partial x.\]
A chain map $f$ of degree $k$ is alternatively a $k$-cycle of $\Hom(C,D)$ or a chain map $s^kC\ra D$ of degree $0$.

\section{Reductions and effective homological algebra}

\subsection*{Definitions}

By a \emph{reduction} (or a strong deformation retraction) $(\alpha,\beta,\eta)\colon C\Ra D$ we will understand
\begin{enumerate}[labelindent=\parindent,leftmargin=2\parindent,label=--]
\item
a pair of chain maps $\alpha\colon C\ra D$, $\beta\colon D\ra C$ of degree $0$, called the \emph{projection} and the \emph{inclusion} respectively, and a map $\eta\colon C\to C$ of degree $1$ called the \emph{homotopy operator} satisfying
\item
$\alpha\beta=\id$, $[\partial,\eta]=\id-\beta\alpha$ (i.e.~$\eta$ is a chain homotopy from $\beta\alpha$ to $\id$) and
\item
$\alpha\eta=0$, $\eta\beta=0$ and $\eta\eta=0$.
\end{enumerate}
The last three conditions will be important later\footnote{On the other hand, it is known that by replacing an arbitrary chain homotopy $\eta$ from $\beta\alpha$ to $\id$ by
\[\big((\id-\beta\alpha)\eta(\id-\beta\alpha)\big)\partial\big((\id-\beta\alpha)\eta(\id-\beta\alpha)\big)\]
the additional relations will start to hold.}.

In our case we will be interested in a special class of reductions which we call \emph{locally effective}. For those both chain complexes and all maps have to be locally effective, where:
\begin{enumerate}[labelindent=\parindent,leftmargin=2\parindent,label=--]
\item
the local effectivity of a chain complex means that one is able to represent its elements in a computer and there are algorithms provided that compute all the relevant operations --- the addition, scalar multiplication, and the differential,
\item
the local effectivity of a (not necessarily chain) map means that there are algorithms provided which compute the value on an arbitrary element.
\end{enumerate}

In the applications we will need a more general concept than that of a reduction. A \emph{strong equivalence} $C\LRa D$ is a span of reductions $C\La\widetilde C\Ra D$. Again, we will be interested in locally effective strong equivalences --- those, where both reductions are locally effective.

\subsection*{The statement of the main theorem}\label{s:statement}
We are now ready to state our main theorem.

\begin{theorem}\label{t:main_theorem}
There exists an algorithm which, given a locally effective $\mcG$-complex $M$ and a locally effective strong equivalence $M\LRa N$, constructs a $\mcG$-linear locally effective strong equivalence $M\LRa N'$. When $N$ is locally finite, so is $N'$.
\end{theorem}

This theorem will be used in \cite{aslep} in the following manner. We will be given an infinite, but locally effective $\mcG$-complex $M$, which we would like to compute with equivariantly. A typical example of such a complex is the chain complex of an infinite simplicial set, such as the Eilenberg-MacLane space $K(\pi,n)$, see the corollary below. By non-equivariant considerations, we will able to construct a strong equivalence of $M$ with a locally finite $\bbZ$-complex, making it possible to perform any (co)homological computations with the original complex $M$. By Theorem~\ref{t:main_theorem}, we will obtain a $\mcG$-linear strong equivalence, making it possible to perform even equivariant (co)homological computations.

\begin{corollary}\label{c:main_corollary}
There is an algorithm that, given a finitely generated $\mcG$-module $\pi$ and natural numbers $n$ and $k$, computes $H_k^G(K(\pi,n))$ and $H^k_G(K(\pi,n))$.
\end{corollary}

Both statements will be proved later, Theorem~\ref{t:main_theorem} in Sections~\ref{s:proof_theorem} and Corollary~\ref{c:main_corollary} in Section~\ref{s:proof_corollary}.

\subsection*{Modules over differential graded algebras}
Let $\mcR$ be a differential graded algebra or, for short, a dga. Let $M$ be a left $\mcR$-module (more precisely differential graded $\mcR$-module), i.e.~a chain complex $M$ equipped with a chain map
\[\mcR\otimes M\to M,\]
satisfying the usual axioms of a module. The chain condition is equivalent to the Leibniz rule
\[\partial(rx)=(\partial r)x+\sign{|r|}r(\partial x)\]
for the scalar multiplication. There is then a left $\mcR$-module structure on $sM$ given by
\[r\cdot sx=\sign{|r|}s(r\cdot x).\]
It is easy to verify that the resulting $\mcR\otimes sM\ra sM$ is really a chain map (while the version with no sign fails to be). An $\mcR$-linear map $f$ of degree $k$ is a map of degree $k$ satisfying
\[r\cdot fx=\sign{|f|\cdot|r|}f(r\cdot x).\]
In particular, $s\colon M\to sM$ is $\mcR$-linear and this fact may serve as the definition of the $\mcR$-module structure on $sM$. Alternatively a map $f\colon M\to N$ of degree $k$ is $\mcR$-linear if and only if the corresponding map $f\colon s^kM\to N$ of degree $0$ is $\mcR$-linear (i.e.~commutes with the action of $\mcR$ in the non-graded sense).

\subsection*{Constructing reductions}
We will now rephrase the conditions on a (locally effective) reduction --- our version is a considerable weakening that proves useful when working equivariantly.

The chain complexes in this section will be modules over a differential graded algebra $\mcR$ and all maps will be assumed to be $\mcR$-linear but not necessarily chain maps. First, let there be given an $\mcR$-linear reduction $(\alpha,\beta,\eta)\colon C\Ra D$. Then, by $\alpha\eta=0$ and $\eta\beta=0$, one may think of $\eta$ as a map
\[\eta'\colon C/\im\beta\cong\ker\alpha\lra\ker\alpha\]
of degree $1$. As such, the condition $[\partial,\eta]=\id-\beta\alpha$ is translated into $[\partial,\eta']=\id$, i.e.~$\eta'$ is a contraction of $\ker\alpha$. On the other hand it is possible to construct a homotopy $\eta$ from any contraction $\eta'$ of $\ker\alpha$ by projecting to $\ker\alpha$ via $\id-\beta\alpha$, i.e.
\[\eta=\eta'(\id-\beta\alpha).\]
We have thus shown so far that a reduction can be specified by a chain map $\alpha$, its section $\beta$ and a contraction of $\ker\alpha$ (and we stress here that the chain complex $\ker\alpha$ does not depend on the section $\beta$).

It is well known that a chain complex of projective modules admits a contraction if and only if it is acyclic. An analogous result holds in the algorithmic setup, once we define all the required notions. We will say that a locally effective $\mcR$-module $F$ is \emph{free as a graded $\mcR$-module} if it is provided with an algorithm that expresses its elements as (unique) combinations of some fixed homogeneous basis\footnote{Typically the elements of $F$ are represented in a computer directly as such combinations.}. The basis is not required to be compatible with the differential --- in effect, $F$ is free as a graded $\mcR$-module and not as a (differential graded) $\mcR$-module. Similarly, we will say that a locally effective $\mcR$-module is \emph{projective as a graded $\mcR$-module} if it is equipped with a locally effective (i.e.~computable) retraction $\id\colon P\xra{i}F\xra{p}P$ from some $\mcR$-module $F$ that is free as a graded $\mcR$-module. Again, the maps are not assumed to be chain maps, but are required to be $\mcR$-linear.

The projective modules have the following property: whenever there is given an ``algorithmically certified'' surjection $A\to B$, i.e.~a map together with an algorithm that computes for each element of the codomain $B$ its (arbitrary) preimage, there exists an algorithm that computes a lift in any diagram
\[\xymatrix{
& A \ar@{->>}[d] \\
P \ar@{-->}[ru] \ar[r] & B
}\]
This lift is computed through the retraction. Namely, one computes a lift of the composition $F\xra{p}P\ra B$ by specifying its values on the basis using the algorithmic (set-theoretic) section and then the resulting lift $F\ra A$ is composed with the inclusion $P\xra{i}F\ra A$ to obtain a lift in the original diagram. This is of course very classical, but we wanted to point out that the same idea works, with correct definitions, also in the algorithmic setup.

We return now to the relationship between acyclicity and contractibility. Let $C$ be an $\mcR$-module, projective as a graded $\mcR$-module, which is ``algorithmically acyclic'': this means that there exists an algorithm that computes, for each cycle $z\in C$, some $c\in C$ with the property $z=\partial c$. Then one can construct a contraction $\sigma$ of $C$ recursively. For simplicity, we assume that $C$ is itself free as a graded $\mcR$-module with $\mcB_n$ the part of the basis of degree $n$. We assume that $\sigma$ is already defined on the $\mcR$-submodule $C^{(n-1)}$ generated by $\mcB_0\cup\cdots\cup\mcB_{n-1}$, and satisfies $[\partial,\sigma]=\id$. Since $\partial\mcB_n\subseteq C^{(n-1)}$, the mapping $\id-\sigma\partial$ is defined on $\mcB_n$ and we may compute a lift in
\[\xymatrix@C=30pt{
& C_{n+1} \ar@{->>}[d]^-\partial \\
\mcB_n \ar@{-->}[ru]^-\sigma \ar[r]_-{\id-\sigma\partial} & Z_n
}\]
by the algorithm for a section of $\partial$ provided by the acyclicity of $C$. We then extend $\sigma$ from $C^{(n-1)}\cup\mcB_n$ uniquely to an $\mcR$-linear map defined on $C^{(n)}$. Since both $[\partial,\sigma]=\partial\sigma+\sigma\partial$ and $\id$ are $\mcR$-linear and agree on $C^{(n-1)}\cup\mcB_n$ they agree on $C^{(n)}$. This finishes the induction. We have thus almost finished the proof of the following technical lemma.

\begin{lemma}\label{l:weak_reduction}
Let $\alpha\colon C\ra D$ be an $\mcR$-linear chain map for which the following hold.
\begin{itemize}
\item
As a graded $\mcR$-module, $C$ is free (or more generally projective).
\item
The map $\alpha\colon C\ra D$ is provided with a locally effective $\mcR$-linear section $\beta_0\colon D\ra C$ (and which needs not be a chain map).
\item
There is an algorithm that computes, for each cycle $z$ of $\ker\alpha$, a chain $\eta_0z\in\ker\alpha$ with the property $\partial\eta_0z=z$.
\end{itemize}
Then from the above data one can construct an $\mcR$-linear reduction $(\alpha,\beta,\eta)\colon C\Ra D$.
\end{lemma}

\begin{proof}
First we observe that $\ker\alpha$ is projective as a graded $\mcR$-module --- it retracts off $C$ with the projection $C\to\ker\alpha$ given by $\id-\beta_0\alpha$. Thus, by the above, one may construct a contraction $\eta'$ of $\ker\alpha$ from the algorithm $\eta_0$. The only remaining step is to construct a section $\beta$ that is a chain map. We set
\[\beta=\beta_0-\eta'[\partial,\beta_0]\]
which is well defined as $[\partial,\beta_0]$ takes values in $\ker\alpha$ by
\[\alpha[\partial,\beta_0]=[\partial,\underbrace{\alpha\beta_0}_{\id}]-\underbrace{[\partial,\alpha]}_0\beta_0=[\partial,\id]=0.\]
As $\eta'$ also takes values in $\ker\alpha$ we have $\alpha\beta=\alpha\beta_0=\id$. Finally, $\beta$ is a chain map by
\[[\partial,\beta]=[\partial,\beta_0]-\underbrace{[\partial,\eta']}_{\id}[\partial,\beta_0]+\eta'\underbrace{[\partial,[\partial,\beta_0]]}_0=0.\]
\end{proof}

\section{Bar construction}

\subsection*{A useful sign convention}
We will be using in the proceeding the following abbreviation. When $x_k$ are elements of a graded abelian group, we denote $\degree xij=|x_i|+\cdots+|x_j|$.

\subsection*{Bar construction}
Let $M$ be a left $\mcR$-module and $N$ a right $\mcR$-module and consider the following graded abelian group
\[B(\mcR,\mcR,M)=\bigoplus_{m\geq 0}\mcR\otimes(s\mcR)^{\otimes m}\otimes M\]
whose elements we write as $r_0|r_1|\cdots|r_m\otimes x$ (the bar $|$ is a shorthand for $\otimes s$) and with the differential $\partial=\partial^\otimes+\partial^\mathrm{alg}$, where $\partial^\otimes=\partial_0^\otimes+\cdots+\partial_{m+1}^\otimes$ for the operators
\begin{align*}
\partial_k^\otimes(r_0|\cdots|r_m\otimes x) & =\sign{k+\degree r0{k-1}}r_0|\cdots|\partial r_k|\cdots|r_m\otimes x \\
\partial_{m+1}^\otimes(r_0|\cdots|r_m\otimes x) & =\sign{m+\degree r0m}r_0|\cdots|r_m\otimes\partial x,
\end{align*}
with $0\leq k\leq m$ ($\partial^\otimes$ is the differential on the tensor product $\mcR\otimes(s\mcR)^{\otimes m}\otimes M$), and where $\partial^\mathrm{alg}=\partial_0^\mathrm{alg}+\cdots+\partial_{m+1}^\mathrm{alg}$ for the operators
\begin{align*}
\partial_k^\mathrm{alg}(r_0|\cdots|r_m\otimes x) & =\sign{k-1+\degree r0{k-1}}r_0|\cdots|r_{k-1}r_k|\cdots|r_m\otimes x \\
\partial_{m+1}^\mathrm{alg}(r_0|\cdots|r_m\otimes x) & =\sign{m+\degree r0{m-1}}r_0|\cdots|r_{m-1}\otimes r_mx,
\end{align*}
with $0\leq k\leq m$ (the index $\mathrm{alg}$ stands for ``algebraic'').

We remark that it is more customary to suspend $M$ too but this convention produces horrible signs later on. The reason is that the above bar construction $B(\mcR,\mcR,M)$ will codify, after perturbing its differential, the action of $\bigoplus_{m\geq 0}\mcR\otimes(s\mcR)^{\otimes m}$ (which we will make into an algebra in the next section) on $M$ rather than on $sM$.

We define the augmentation map $\varepsilon\colon B(\mcR,\mcR,M)\to M$ by $\varepsilon(r_0\otimes x)=r_0x$ and by sending all longer tensors to zero, $\varepsilon(r_0|\cdots|r_m\otimes x)=0$ for $m\geq 1$.

\begin{theorem} \label{t:resolution_of_G_modules}
Suppose that, as a graded $\mcR$-module, $M$ is free. Then the augmentation map $\varepsilon\colon B(\mcR,\mcR,M)\ra M$ is a projection of an $\mcR$-linear reduction.
\end{theorem}

\begin{proof}
By Lemma~\ref{l:weak_reduction}, we need to construct a section and a non-equivariant contraction of $\ker\varepsilon$ (which is even stronger than the requested algorithm). To define a section $\zeta_0$ start with some $\mcR$-basis of $M$ and specify $\zeta_0(x)$ on a basis element $x$ by $\zeta_0(x)=1\otimes x\in BM_0$.

A non-equivariant contraction of $\ker\varepsilon$ is given by
\[\eta_0\colon r_0|\cdots|r_m\otimes x\mapsto 1|r_0|\cdots|r_m\otimes x.\]
It is obvious that $[\partial^\otimes,\eta_0]=0$ and that $[\partial^\mathrm{alg},\eta_0]z=z$ for all $z$ of length $m>0$. Let finally $z$ be of length $m=0$. Then $\partial^\mathrm{alg}\eta_0z=z+1\otimes\varepsilon z$. Thus on elements of $\ker\varepsilon$ of length $0$ we also obtain $[\partial^\mathrm{alg},\eta_0]=\partial^\mathrm{alg}\eta_0=\id$.
\end{proof}

\begin{remark}
The same is true when $M$ is merely projective as a graded $\mcR$-module.
\end{remark}

\section{Homotopy $\mcR$-modules}\label{s:homotopy_ZG_modules}

It is well known that the structure of a module over a differential graded algebra is homotopy invariant, i.e.~passes to homotopy equivalent chain complexes, when the dga in question is cofibrant (see e.g.~\cite{Markl}). In our applications, we are interested in modules over the dga $\mcG=\bbZ G$ which is not cofibrant. We will therefore be interested in its cofibrant replacement, which we will call $\mcG_\infty$. Any chain complex of $\mcG$-modules will then automatically be a $\mcG_\infty$-module and this structure will pass to all homotopy equivalent chain complexes.

Since we are interested in computations with these modules, a mere existence is not sufficient. We will therefore not need to prove that $\mcG_\infty$ is indeed a cofibrant replacement of $\mcG$ but we will concentrate on algorithms for the transfer of the structure of a $\mcG_\infty$-module. In this section we introduce more generally, for an essentially arbitrary dga $\mcR$, its replacement $\mcR_\infty$. In the next section we continue with describing the transport of the structure along homotopy equivalences (reductions).

\subsection*{The differential graded algebra $\mcR_\infty$}
Let $\mcR$ be a differential graded algebra which is free as a graded abelian group. Its basis elements will be called the generators of $\mcR$ and we assume that the unit $1$ of the algebra is one of them. We will now describe its replacement $\mcR_\infty=\Omega B\mcR$. As an associative unital graded algebra, it is generated by the graded abelian group $\bigoplus_{m\geq 0}\mcR\otimes(s\mcR)^{\otimes m}$ with simple tensors in $\mcR\otimes(s\mcR)^{\otimes m}$ denoted by $(r_0,\ldots,r_m)$; the dimension of this generator is $m+\degree r0m$. The differential is given by the formula
\begin{align}
\partial(r_0,\ldots,r_m) &=\sum_{k=0}^m\sign{k+\degree r0{k-1}}(r_0,\ldots,\partial r_k,\ldots,r_m) \tag{$\partial^\otimes$}\\
& \ \ \ +\sum_{k=1}^m\sign{k-1+\degree r0{k-1}}(r_0,\ldots,r_{k-1}r_k,\ldots,r_m) \tag{$\partial^+$}\\
& \ \ \ +\sum_{k=1}^m\sign{k+\degree r0{k-1}}(r_0,\ldots,r_{k-1})\cdot(r_k,\ldots,r_m) \tag{$\partial^-$}
\end{align}
It is easy to see that $\partial$ has degree $-1$ and is indeed a differential. We denote its first term by $\partial^\otimes$ and the remaining two by $\partial^\mathrm{alg}=\partial^++\partial^-$. The ideal of relations is generated by $(1)-1$ and by all $(r_0,\ldots,r_m)$ with at least one $r_i=1$. By an easy calculation, this ideal is closed under $\partial$ and $\mcR_\infty$ is defined as the quotient by this ideal.\footnote{By the form of the differential it is clear that $\mcR_\infty$ is a cellular dga: it is generated by $(r_0,\ldots,r_m)$ with $r_0,\ldots,r_m$ generators of $\mcR$, none of which is $1$ and may be added according to their dimension and glued by their boundary. In particular, $\mcR_\infty$ is indeed cofibrant.}

There is an alternative description in the case that $\mcR$ is augmented --- in this case $\mcR_\infty$ is, as an associative unital graded algebra, the tensor algebra of $\bigoplus_{m\geq 0}\overline\mcR\otimes(s\overline\mcR)^{\otimes m}$ where $\overline\mcR$ denotes the augmentation ideal. Since $\overline\mcR$ is a differential ideal, the above formula yields a well-defined differential on this tensor algebra.

\subsection*{The relation of $\mcR_\infty$ to $\mcR$}
There is an evident dga-map $\mcR_\infty\to\mcR$ sending $(r)$ to $r$ and the remaining generators to~$0$. It admits an obvious section $\mcR\ra\mcR_\infty$ which is only a chain map --- it does not respect the multiplication.

The algebra $\mcR_\infty$ has a natural filtration by subcomplexes $\mcR^d$ which are formed by elements of length at most $d$ where the length of a product is
\[\ell(\rho_1\cdot\cdots\cdot\rho_n)=\ell\rho_1+\cdots+\ell\rho_n\]
and the length of a generator is $\ell(r_0,\ldots,r_m)=m+1$. Clearly one has $\mcR^d\cdot\mcR^e\subseteq\mcR^{d+e}$ and $\mcR_\infty=\bigcup_d\mcR^d$.

\begin{theorem} \label{t:finite_approximation}
The map $\mcR^d\ra\mcR$ is a projection of a reduction for all $d\geq 1$.
\end{theorem}

\begin{proof}
For $d=1$ the map is an isomorphism. The contraction $\eta_d'$ of the quotient $\mcR^d/\mcR^{d-1}$ is given by
\[(r)\cdot(r_0,\ldots,r_m)\cdot\rho\longmapsto\sign{|r|+1}(r,r_0,\ldots,r_m)\cdot\rho\]
if the first factor has length $1$ (and is not the sole factor), while the contraction is defined to be $0$ on the remaining additive generators.

One may then define a homotopy $\eta_d$ on $\mcR^d$ by extending the above to the generators of $\mcR^{d-1}$ by $0$. It is a homotopy of $\id$ with some map $p_d=\id-[\partial,\eta_d]\colon \mcR^d\ra\mcR^{d-1}$. The deformation of $\mcR^d$ is then given as
\[\eta_d+\eta_{d-1}(\id-[\partial,\eta_d])+\cdots+\eta_2(\id-[\partial,\eta_3])\cdots(\id-[\partial,\eta_d]),\]
clearly a homotopy between $\id$ and the projection $p_2\cdots p_d\colon \mcR^d\ra\mcR^1\cong\mcR$.
\end{proof}

\begin{tremark}
It is very simple to compute $p_d=\id-[\partial,\eta_d]$ directly and thus to simplify the computation of the overall contraction. The value on $(r)\cdot(r_0,\ldots,r_m)\cdot\rho$ is $(rr_0,r_1,\ldots,r_m)\cdot\rho$, the value on $(r,s)\cdot(r_0,\ldots,r_m)\cdot\rho$ is $\sign{|s|}(rs,r_0,\ldots,r_m)\cdot\rho$ and $p_d$ is zero otherwise.
\end{tremark}

\begin{corollary}\label{c:reduction_G_ZG}
The map $\mcR_\infty\ra\mcR$ is a projection of a reduction.\qed
\end{corollary}

\section{Transfer of the structure}

In this section, we will describe how a structure of an $\mcR_\infty$-module is transported along a reduction. There are two directions, which we call ``easy'' and ``basic'' in accordance with the easy and basic perturbation lemmas of homological perturbation theory, see e.g.~\cite[Section 4.8]{Sergeraert}.

\subsection*{The easy case}
We assume here, that $\mcR$ is augmented\footnote{More generally, when there is given an arbitrary $\bbZ$-linear map $\varepsilon\colon\mcR_0\to\bbZ$ satisfying $\varepsilon(1)=1$ (which exists by our assumption of freeness of $\mcR_0$), we may set $\rho x=\beta\rho\alpha x$ when $\rho=(r_0,\ldots,r_m)$ with $m\geq 2$ and
\begin{align*}
(r)x & =\beta(r)\alpha x+\varepsilon(r)\cdot(\id-\beta\alpha)x+\varepsilon(\partial r)\cdot\eta x, \\
(r,s)x & =\beta(r,s)\alpha x+(\varepsilon(rs)-\varepsilon(r)\cdot\varepsilon(s))\cdot\eta x.
\end{align*}}. Let $(\alpha,\beta,\eta)\colon M\Ra N$ be a reduction and let $N$ be equipped with a structure of an $\mcR_\infty$-module. Then we define an $\mcR_\infty$-module structure on $M$ by $\rho x=\beta\rho\alpha x$, whenever $\rho=(r_0,\ldots,r_m)\in\mcR_\infty$ with all $r_i$ in the augmentation ideal. Since the augmentation ideal is closed under $\partial$, the Leibniz rule
\[\partial(\rho x)=\beta(\partial\rho)\alpha x+\sign{|\rho|}\beta\rho\alpha(\partial x)=(\partial\rho)x+\sign{|\rho|}\rho(\partial x)\]
holds for $\rho$. When some $r_i$ is a multiple of $1$, the action is given by the axioms of an $\mcR_\infty$-module and the Leibniz rule is automatically satisfied for such $\rho$. Therefore $M$ is indeed an $\mcR_\infty$-module.

All the maps $\alpha$, $\beta$ and $\eta$ are $\mcR_\infty$-linear and thus $M\Ra N$ is in fact an $\mcR_\infty$-linear reduction. We will explain in Section~\ref{s:R_linear_maps} how to replace this reduction by an $\mcR$-linear one in a more general context which applies also to the transfer in the opposite direction.

\subsection*{The basic case}
Let $(\alpha,\beta,\eta)\colon M\Ra N$ be a reduction and let $M$ be equipped with a structure of an $\mcR_\infty$-module. We first define the following family of maps $M\ra M$
\[\Sh(r_0,\ldots,r_m)x=\sum_{\makebox[30pt]{$\scriptstyle\substack{n\geq 0,\\ 0<k_1<\cdots<k_n<m+1}$}}(r_0,\ldots,r_{k_1-1})\eta\cdots\eta(r_{k_n},\ldots,r_m)x\]
(the shuffles of $(r_0,\ldots,r_m)$ and $\eta$). The corresponding family of maps $N\ra N$ is given by
\[(r_0,\ldots,r_m)y=\alpha\Sh(r_0,\ldots,r_m)\beta y.\]
The following is the main result of this section.

\begin{theorem}
The above prescription defines an action of $\mcR_\infty$ on $N$.
\end{theorem}

\begin{proof}
Lemma~\ref{l:differential_Sh} below gives a formula for the differential of $\Sh(r_0,\ldots,r_m)$. Decorating it with chain maps $\alpha$ and $\beta$ (i.e.~$[\partial,\alpha]=0$ and $[\partial,\beta]=0$) the result is easily obtained.
\end{proof}

\begin{lemma}\label{l:differential_Sh}
The differential $[\partial,\Sh(r_0,\ldots,r_m)]$ equals
\begin{align*}
\Sh\partial^\otimes(r_0,\ldots,r_m) & +\sum_{k=1}^m\sign{k-1+\degree r0{k-1}}\Sh(r_0\ldots,r_{k-1}r_k,\ldots,r_m) \\
& +\sum_{k=1}^m\sign{k+\degree r0{k-1}}\Sh(r_0,\ldots,r_{k-1})\beta\alpha\Sh(r_k,\ldots,r_m).
\end{align*}
\end{lemma}

\begin{proof}
This is a relatively simple computation:
\begin{align*}
[\partial,\Sh(r_0,\ldots,r_m)] & =\sum_{k=1}^m\sign{k-1+\degree r0{k-1}}\Sh(r_0,\ldots,r_{k-1})[\partial,\eta]\Sh(r_k,\ldots,r_m) \\
& +\sum_{\makebox[0pt]{$\scriptstyle 0\leq i\leq j\leq m$}}\sign{i+\degree r0{i-1}}\Sh(r_0,\ldots,r_{i-1})\eta[\partial,(r_i,\ldots,r_j)]\eta\Sh(r_{j+1},\ldots,r_m) \\
\end{align*}
(in the case $i=0$ and/or $j=m$ the term on the left and/or right of $[\partial,(r_i,\ldots,r_j)]$ is to be left out). The first term equals
\begin{align*}
& \sum_{k=1}^m\sign{k-1+\degree r0{k-1}}\Sh(r_0,\ldots,r_{k-1})\Sh(r_k,\ldots,r_m) \\
& +\sum_{k=1}^m\sign{k+\degree r0{k-1}}\Sh(r_0,\ldots,r_{k-1})\beta\alpha\Sh(r_k,\ldots,r_m)
\end{align*}
while the second is
\begin{align*}
& \phantom{{}={}}\sum_{\makebox[0pt]{$\scriptstyle 0\leq i\leq k\leq j\leq m$}}\sign{k+\degree r0{k-1}}\Sh(r_0,\ldots,r_{i-1})\eta(r_i,\ldots,\partial r_k,\ldots,r_j)\eta\Sh(r_{j+1},\ldots,r_m) \\
& \ \ \ +\sum_{\makebox[0pt]{$\scriptstyle 0\leq i<k\leq j\leq m$}}\sign{k-1+\degree r0{k-1}}\Sh(r_0,\ldots,r_{i-1})\eta(r_i,\ldots,r_{k-1}r_k,\ldots,r_j)\eta\Sh(r_{j+1},\ldots,r_m) \\
& \ \ \ +\sum_{\makebox[0pt]{$\scriptstyle 0\leq i<k\leq j\leq m$}}\sign{k+\degree r0{k-1}}\Sh(r_0,\ldots,r_{i-1})\eta(r_i,\ldots,r_{k-1})(r_k,\ldots,r_j)\eta\Sh(r_{j+1},\ldots,r_m) \\
& =\Sh(\partial^\otimes(r_0,\ldots,r_m))+\sum_{k=1}^m\sign{k-1+\degree r0{k-1}}\Sh(r_0,\ldots,r_{k-1}r_k,\ldots,r_m) \\
& \phantom{{}=\Sh(\partial^\otimes(r_0,\ldots,r_m))} \ \ \ +\sum_{k=1}^m\sign{k+\degree r0{k-1}}\Sh(r_0,\ldots,r_{k-1})\Sh(r_k,\ldots,r_m).
\end{align*}
Adding these together and cancelling the equal terms yields the desired formula.
\end{proof}

\section{Strictification of $\mcR_\infty$-modules}
Consider an $\mcR_\infty$-module $M$. We will define its resolution, which will be an $\mcR$-module. When $M$ is an $\mcR$-module, it is the bar construction $B(\mcR,\mcR,M)$ and in the general case, we have to accomodate the differential to the fact that $M$ does not have a strict action of $\mcR$. What this means is that
\[(BM)_m=\mcR\otimes(\mcR)^{\otimes m}\otimes M\]
is only a ``homotopy coherent'' semi-simplicial object (the simplicial identities do not hold strictly, but only up to a coherent system of higher order homotopies). We will not give details here of how this structure can be described explicitly as it turns out, that one may strictify this diagram in a simple way and get a sort of cubical diagram whose geometric realization we will now describe.

Concretely, on the graded abelian group
\[BM=\bigoplus_m\mcR\otimes (s\mcR)^{\otimes m}\otimes M,\]
consider the operators (where we use $|$ instead of $\otimes s$ as usual to increase readability)
\begin{align*}
\partial_k^\otimes(r_0|\cdots|r_m\otimes x) & =\sign{k+\degree r0{k-1}}r_0|\cdots|\partial r_k|\cdots|r_m\otimes x, && 0\leq k\leq m \\
\partial_{m+1}^\otimes(r_0|\cdots|r_m\otimes x) & =\sign{m+\degree r0m}r_0|\cdots|r_m\otimes \partial x \\
\partial_k^+(r_0|\cdots|r_m\otimes x) & =\sign{k-1+\degree r0{k-1}}r_0|\cdots|r_{k-1}r_k|\cdots|r_m\otimes x, && 1\leq k\leq m \\
\partial_k^-(r_0|\cdots|r_m\otimes x) & =\sign{k+\degree r0{k-1}}r_0|\cdots|r_{k-1}\otimes (r_k,\ldots,r_m)x, && 1\leq k\leq m.
\end{align*}
of degree $-1$. We define $\partial^\otimes=\partial_0^\otimes+\cdots+\partial_{m+1}^\otimes$ and similarly $\partial^+=\partial_1^++\cdots+\partial_m^+$ and $\partial^-=\partial_1^-+\cdots+\partial_m^-$. The differential $\partial^\otimes$ is that of the tensor product $\mcR\otimes (s\mcR)^{\otimes m}\otimes M$. Finally, we define the differential on $BM$ as
\[\partial=\partial^\otimes+\partial^++\partial^-.\]

\begin{lemma}
The operator $\partial$ is a differential, $\partial^2=0$.
\end{lemma}

\begin{proof}
It is easy to verify the following relations for $k<\ell$
\[\partial^\otimes\partial^\otimes=0, \ \ \ \ 
\partial_k^+\partial_\ell^++\partial_{\ell-1}^+\partial_k^+=0, \ \ \ \ 
\partial_\otimes \partial_\ell^++\partial_\ell^+\partial_\otimes=0\]
yielding $(\partial^\otimes+\partial^+)^2=0$. Moreover, one has for $k<\ell$
\begin{align*}
\partial_k^-\partial_\ell^-(r_0|\cdots|r_m\otimes x) & =\sign{k+\ell+\degree rk{\ell-1}}r_0|\cdots|r_{k-1}\otimes (r_k,\ldots,r_{\ell-1})(r_\ell,\cdots,r_m)x \\
\partial_k^-\partial_\ell^+(r_0|\cdots|r_m\otimes x) & =\sign{k+\ell-1+\degree rk{\ell-1}}r_0|\cdots|r_{k-1}\otimes (r_k,\ldots,r_{\ell-1}r_\ell,\cdots,r_m)x
\end{align*}
while for $k>\ell$ the following hold
\begin{align*}
(\partial_k^-\partial_\ell^++\partial_\ell^+\partial_k^-)(r_0|\cdots|r_m\otimes x) & =0 \\
(\partial_k^-\partial_\otimes+\partial_\otimes \partial_k^-)(r_0|\cdots|r_m\otimes x) & =-r_0|\cdots|r_{k-1}\otimes (\partial^\mathrm{alg}(r_k,\ldots,r_m))x
\end{align*}
which sum up to $\partial_k^-\partial+(\partial^\otimes+\partial^+)\partial_k^-=0$. Summing up over $k$ and with the previous $(\partial^\otimes+\partial^+)^2=0$, we finally obtain $\partial^2=0$.
\end{proof}

There are obvious chain maps $\varepsilon_0\colon BM\to M$, sending $r_0|\cdots|r_m\otimes x$ to $(r_0,\ldots,r_m)x$, and $\zeta_0\colon M\to BM$, sending $x$ to $1\otimes x$. Together with the homotopy operator
\[\eta_0(r_0|\cdots|r_m\otimes x)=1|r_0|\cdots|r_m\otimes x,\]
they expresses $M$ as a deformation retract of $BM$. There is no sense in speaking about any equivariancy here --- $M$ is an $\mcR_\infty$-complex, while $BM$ is an $\mcR$-complex and none of these maps is $\mcR_\infty$-linear.

In the special case of an $\mcR$-module $M$, the projection $\varepsilon_0\colon BM\to M$ is $\mcR$-linear --- in fact, $BM=B(\mcR,\mcR,M)$ and $\varepsilon_0$ is the augmentation of Theorem~\ref{t:resolution_of_G_modules}. As observed in that theorem, it is part of an $\mcR$-linear reduction. This will be important later.

\section{Homotopy $\mcR$-linear maps}\label{s:R_linear_maps}

\begin{definition}
An \emph{$\mcR_\infty$-map} $M\to N$ of degree $d$ is a map $f\colon BM\to BN$ of the form
\[f(r_0|\cdots|r_m\otimes x)=\sum_{k=0}^m\sign{d(k+\degree r0k)}r_0|\cdots|r_k\otimes f_{m-k}|r_{k+1}|\cdots|r_m)x\]
for some maps $f_\ell\colon(s\mcR)^{\otimes\ell}\to\Hom(M,N)$ of degree $d$, which we call the \emph{components} of $f$. We will write $f_*\colon M\to N$ to denote the collection of the $f_\ell$.
\end{definition}

We will be mostly interested in $\mcR_\infty$-chain maps of degree $0$ but it is convenient to have also a notion of an $\mcR_\infty$-homotopy.

In the following proposition, we understand $(s\mcR)^{\otimes\ell}$ equipped with the differential $\partial^\otimes$.

\begin{proposition}
Let $f$ be an $\mcR_\infty$-map with components $f_\ell\colon(s\mcR)^{\otimes\ell}\to\Hom(M,N)$ of degree $d$. Then the differential $f'=[\partial,f]$ is an $\mcR_\infty$-map of degree $d-1$ whose components $f'_\ell$ satisfy the equations
\begin{align}
[\partial,f_\ell]|r_1|\cdots|r_\ell) & =f'_\ell|r_1|\cdots|r_\ell)+\sum_{k=1}^\ell\sign{d(k+\degree r1k)}(r_1,\ldots,r_k)\circ f_{\ell-k}|r_{k+1}|\cdots|r_\ell) \nonumber\\
& \ \ \ +\sum_{k=1}^{\ell-1}\sign{k+d+\degree r1k}f_{\ell-1}|r_1|\cdots|r_k\cdot r_{k+1}|\cdots|r_\ell) \nonumber\\
& \ \ \ +\sum_{k=0}^{\ell-1}\sign{k+1+d+\degree r1k}f_k|r_1|\cdots|r_k)\circ(r_{k+1},\ldots,r_\ell),
\tag{$\triangle$}\label{e:G_map}
\end{align}
where the elements of $\mcR_\infty$ in the first and the third row are to be interpreted as their respective images in $\Hom(N,N)$ and $\Hom(M,M)$.

In particular, $f$ is an $\mcR_\infty$-chain map of degree $d$ if and only if the equations \eqref{e:G_map} are satisfied with $f'_m=0$.
\end{proposition}

\begin{tremark}
Another point of view is that the $f_\ell$ are maps $\mcR^{\otimes\ell}\to\Hom(M,N)$ of degree $\ell$. If we adopted different sign conventions on the algebra $\mcR_\infty$ and the bar construction $BM$, these would get the following interpretations: $f_0$ is a map, which preserves the action of $r\in\mcR$ up to homotopy $f_1|r)$ and a coherent system of higher order homotopies $f_2|r_1|r_2)$, etc.
\end{tremark}
 
\begin{proof}
Throughout the proof, we will use maps $\varphi_\ell(r_1|\cdots|r_\ell\otimes x)=f_\ell|r_1|\cdots|r_\ell)x$ of degree $d+1$. Their differentials are related to those of $f_\ell$ in exactly the same manner,
\[[\partial,\varphi_\ell](r_1|\cdots|r_\ell\otimes x)=[\partial,f_\ell]|r_1|\cdots|r_\ell)x.\]
We will also abbreviate $z=r_1|\cdots|r_\ell\otimes x$. One can easily check the following inductive formulas
\begin{align*}
\partial(r_0|z) & =\partial r_0|z-\sign{|r_0|}r_0|\partial z+\sign{|r_0|}r_0z-\sign{|r_0|}r_0\otimes\alpha z, \\
f(r_0|z) & =\sign{d(1+|r_0|)}r_0|fz+\sign{d|r_0|}r_0\otimes\varphi_\mathrm{max}z,
\end{align*}
where $\varphi_\mathrm{max}$ denotes the component with the maximal index; in particular $\varphi_\mathrm{max}z=\varphi_\ell z$. Composing in one direction, we get
\begin{align*}
\partial f(r_0|z) & =\sign{d(1+|r_0|)}\partial r_0|fz+\sign{(d-1)(1+|r_0|)}r_0|\partial fz \\
& \ \ \ -\sign{(d-1)(1+|r_0|)}r_0fz+\sign{(d-1)(1+|r_0|)}r_0\otimes\alpha fz \\
& \ \ \ +\sign{d|r_0|}\partial r_0\otimes\varphi_\mathrm{max}z+\sign{(d-1)|r_0|}r_0\otimes\partial\varphi_\mathrm{max}z,
\end{align*}
while the composition in the opposite direction is
\begin{align*}
f\partial(r_0|z) & =\sign{d|r_0|}\partial r_0|fz+\sign{d(1+|r_0|)}\partial r_0\otimes\varphi_\mathrm{max}z \\
& \ \ \ +\sign{(d-1)(1+|r_0|)}r_0|f\partial z-\sign{(d-1)|r_0|}r_0\otimes\varphi_\mathrm{max}\partial z \\
& \ \ \ +\sign{|r_0|}fr_0z-\sign{(d-1)|r_0|}r_0\otimes f_0\alpha z.
\end{align*}

By its form, $f$ is always $\mcR$-linear, i.e.~$fr_0z=\sign{d|r_0|}r_0fz$. Thus, the corresponding terms in the difference $[\partial,f]=\partial f-\sign{d}f\partial$ cancel out and we obtain
\begin{align*}
& [\partial,f](r_0|z)=\sign{(d-1)(1+|r_0|)}r_0|[\partial,f]z \\
& \ \ \ +\sign{(d-1)|r_0|}r_0\otimes(-\sign{d}\alpha fz+\partial\varphi_\mathrm{max}z-\sign{d+1}\varphi_\mathrm{max}\partial z-\sign{d+1}f_0\alpha z).
\end{align*}
Consequently $f'=[\partial,f]$ is an $\mcR_\infty$-map of degree $d-1$ with components $\varphi'_\ell$ given by
\[\varphi'_\ell=-\sign{d}\alpha f+\partial\varphi_\ell-\sign{d+1}\varphi_\mathrm{max}\partial-\sign{d+1}f_0\alpha\]
The differential from the statement equals
\begin{align*}
[\partial,\varphi_\ell] & =\partial\varphi_\ell-\sign{d+1}\varphi_\ell\partial^\otimes \\
& =(\partial\varphi_m-\sign{d+1}\varphi_\mathrm{max}\partial)+\sign{d+1}\varphi_\mathrm{max}(\partial^++\partial^-).
\end{align*}
Expressing the first term from the formula for $\varphi'_\ell$ we obtain
\[[\partial,\varphi_\ell]=\varphi'_\ell+\sign{d}\alpha f+\sign{d+1}\varphi_{\ell-1}\partial^++\sign{d+1}(\varphi_\mathrm{max}\partial^-+f_0\alpha).\]
The terms of this equation correspond exactly to the terms of \eqref{e:G_map}.
\end{proof}

The components of the composition $gf$ of two $\mcR_\infty$-maps $f$ and $g$ are easily seen to be
\[(gf)_\ell|r_1|\cdots|r_\ell)=\sum_{k=0}^\ell\sign{|f|(k+\degree r1k)}g_k|r_1|\cdots|r_k)\circ f_{\ell-k}|r_{k+1}|\cdots|r_\ell).\]
The composition is associative with unit $\id\colon BM\to BM$, whose components are $\id_0=\id$ and $\id_\ell=0$ for all $\ell>0$, see also the first example below.

\begin{example}\hskip0pt
\begin{enumerate}[labelindent=\parindent,leftmargin=2\parindent,label=--]
\item
Any $\mcR_\infty$-linear map $f_0\colon M\to N$ extends to an $\mcR_\infty$-map $f_*\colon M\to N$ by $f_\ell=0$, for all $\ell>0$. This is clear from \eqref{e:G_map}.

\item
The projection $\varepsilon_0\colon BM\to M$ can be made into an $\mcR_\infty$-map $\varepsilon_*\colon BM\to M$ by
\[\varepsilon_\ell|r_1|\cdots|r_\ell)(r'_0|\cdots|r'_m\otimes x)=\sign{\ell+\degree r1\ell}(r_1,\ldots,r_\ell,r'_0,\ldots,r'_m)x.\]
The inclusion $\zeta_0\colon M\to BM$ can be made into an $\mcR_\infty$-map $\zeta_*\colon M\to BM$ by
\[\zeta_\ell|r_1|\cdots|r_\ell)x=1|r_1|\cdots|r_\ell\otimes x.\]
We have $\varepsilon\zeta=\id$, while the other composition $\zeta\varepsilon$ is $\mcR_\infty$-homotopic to $\id$ via the $\mcR_\infty$-homotopy $\eta_*$ with components
\[\eta_\ell|r_1|\cdots|r_\ell)(r'_0|\cdots|r'_m\otimes x)=\sign{\ell+\degree r1\ell}1|r_1|\cdots|r_\ell|r'_0|\cdots|r'_m\otimes x.\]
Put together, they give an $\mcR_\infty$-reduction $(\varepsilon_*,\zeta_*,\eta_*)\colon BM\to M$.
\end{enumerate}
\end{example}

\subsection*{Transfer of the structure}
Suppose now that $M$ is an $\mcR_\infty$-module and that there is given a reduction $(\alpha_0,\beta_0,\eta)\colon M\Ra N$. We will extend $\alpha_0$ and $\beta_0$ to $\mcR_\infty$-maps $\alpha_*$ and $\beta_*$. With a bit of extra work on the homotopy operator in the following section, we will obtain an $\mcR$-linear reduction $BM\Ra BN$.

The respective components of $\alpha_*$ and $\beta_*$ are
\begin{align*}
\alpha_\ell|r_1|\cdots|r_\ell)x & =\sign{\ell+\degree r1\ell}\alpha_0\Sh(r_1,\ldots,r_\ell)\eta x \\
\beta_\ell|r_1|\cdots|r_\ell)y & =\eta\Sh(r_1,\ldots,r_\ell)\beta_0 y
\end{align*}
We will now show, that $\beta_*$ is indeed an $\mcR_\infty$-map, leaving $\alpha_*$ to the reader. We need to verify the equations \eqref{e:G_map}. Thus, we compute
\begin{align*}
[\partial,\beta_\ell]|r_1|\cdots|r_\ell) & =[\partial,\beta_\ell|r_1|\cdots|r_\ell)]-\beta_\ell\partial^\otimes|r_1|\cdots|r_\ell) \\
& =[\partial,\eta]\Sh(r_1,\ldots,r_\ell)\beta_0-\eta[\partial,\Sh(r_1,\ldots,r_\ell)]\beta_0 \\
& \ \ \ +\eta(\Sh\partial^\otimes(r_1,\ldots,r_\ell))\beta_0
\end{align*}
which equals, by the Leibniz rule and Lemma~\ref{l:differential_Sh}, to the sum
\begin{align*}
(\id-\beta_0\alpha_0)\Sh(r_1,\ldots,r_\ell)\beta_0 & -\sum_{k=1}^{\ell-1}\sign{k-1+\degree r1k}\underbrace{\eta\Sh(r_1\ldots,r_kr_{k+1},\ldots,r_\ell)\beta_0}_{\beta_{\ell-1}|r_1|\cdots|r_kr_{k+1}|\cdots|r_\ell)} \\
& {}-\sum_{k=1}^{\ell-1}\sign{k+\degree r1k}\underbrace{\eta\Sh(r_1,\ldots,r_k)\beta_0}_{\beta_k|r_1|\cdots|r_k)}\underbrace{\alpha_0\Sh(r_{k+1},\ldots,r_\ell)\beta_0}_{(r_{k+1},\ldots,r_\ell)}
\end{align*}
The first term is $\Sh(r_1,\ldots,r_\ell)\beta_0-\beta_0(r_1,\ldots,r_\ell)$. The computation is finished by the observation
\[\Sh(r_1,\ldots,r_\ell)\beta_0=(r_1,\ldots,r_\ell)\beta_0+\sum_{k=1}^{\ell-1}(r_1,\ldots,r_k)\underbrace{\eta\Sh(r_{k+1},\ldots,r_\ell)\beta_0}_{\beta_{\ell-k}|r_{k+1}|\cdots|r_\ell)}.\]

It is easy to see that the composition $\alpha\beta$ equals $\id$ since the only non-zero composite is $\alpha_0\beta_0=\id$, thanks to the identities $\alpha_0\eta=0$, $\eta\beta_0=0$ and $\eta\eta=0$.

\subsection*{Constructing the homotopy operator for $BM\Ra BN$}
The homotopy operator $\eta$ of the reduction $M\Ra N$ is not an $\mcR_\infty$-map in general. An $\mcR$-linear homotopy operator for $BM\Ra BN$ is constructed by Lemma~\ref{l:weak_reduction} from the algorithmic acyclicity of $\ker\alpha$ which we now prove. We assume at this point more generally, that $\alpha$ is an arbitrary $\mcR_\infty$-map.

\begin{proposition}
Let $\alpha\colon BM\ra BN$ be an $\mcR_\infty$-map for which $\ker\alpha_0$ is acyclic. Then $\ker\alpha$ is also acyclic and the same is true in the algorithmic setup.
\end{proposition}

\begin{proof}
The defining formula for $\alpha$,
\[\alpha(r_0|\cdots|r_m\otimes x)=\sum_{k=0}^m r_0|\cdots|r_k\otimes\alpha_{m-k}|r_{k+1}|\cdots|r_m)x,\]
shows that any $z\in\ker\alpha$ has to have its component $z_{\mathrm{max}}$ of maximal length lying in $\mcR\otimes(s\mcR)^{\otimes m}\otimes\ker\alpha_0$. Moreover the component of $\partial z$ of the maximal length $m$ equals $\partial_\otimes z_\mathrm{max}$. Thus, if $z$ is a cocycle, $z_{\mathrm{max}}$ has to be a cocycle with respect to $\partial_\otimes$. Let $\eta$ be a contracting homotopy of $\ker\alpha_0$. Then $\id\otimes\id^{\otimes m}\otimes\eta$ is a contracting homotopy of $\mcR\otimes(s\mcR)^{\otimes m}\otimes\ker\alpha_0$ and we obtain
\begin{align*}
\partial(\id\otimes\id^{\otimes m}\otimes\eta)z_{\mathrm{max}} & =\partial_\otimes(\id\otimes\id^{\otimes m}\otimes\eta)z_{\mathrm{max}}+\textrm{shorter terms} \\
& =z_{\mathrm{max}}-(\id\otimes\id^{\otimes m}\otimes\eta)\partial_\otimes z_{\mathrm{top}}+\textrm{shorter terms} \\
& =z_{\mathrm{max}}+\textrm{shorter terms}.
\end{align*}
Consequently, $z-\partial(\id\otimes\id^{\otimes m}\otimes\eta)z_{\mathrm{max}}$ has length at most $m-1$ and we may finish the computation of a chain $c$ with $z=\partial c$ by induction.
\end{proof}

Combining the above constructions with Lemma~\ref{l:weak_reduction}, we obtain an $\mcR$-linear reduction $BM\Ra BN$.

\subsection*{The proof of Theorem~\ref{t:main_theorem}}\label{s:proof_theorem}
Let $M$ be a $\mcG$-complex and let there be given a span of reductions $M\La\widetilde M\Ra N$. Then we have constructed a span of $\mcG$-linear reductions $BM\La B\widetilde M\Ra BN$ appearing on the right of
\[M\La BM\La B\widetilde M\Ra BN.\]
By the explicit form of the involved operators, it is clear that this construction is algorithmic. The reduction on the left is given by Theorem~\ref{t:resolution_of_G_modules} and the fact that $BM=B(\mcG,\mcG,M)$.

\begin{remark}[on polynomiality]
Let the finite group $G$ be fixed. Assume, that all the algorithms involved in $M$, $N$ (including the action of $G$ on $M$), and in the strong equivalence $M\LRa N$, have running times on an input $x$ bounded by a function $p(|x|,\operatorname{size}x)$ which is polynomial when the dimension $|x|$ of the input is fixed.

It is then clear from our formulas that the same will be true for the strong equivalence $M\LRa BN$. In \cite{polypost}, we define a ``chain complex with polynomial-time homology'' as a parametrized family of strong equivalences as above (the involved polynomials also depend on the parameter), where in addition we require an algorithm that outputs a basis of each $N_n$ with running time bounded by a polynomial function of the parameter. Then, given such a family, the resulting family will have ``polynomial-time equivariant homology'' (since clearly the rank of $(BN)_n$ over $\mcG$ is $\operatorname{rk}(BN)_n=\sum_{m=0}^n|G|^m\operatorname{rk}N_{n-m}$ and thus bounded by a polynomial).
\end{remark}

\section{The equivariant (co)homology of Eilenberg-MacLane spaces}\label{s:proof_corollary}

\begin{proof}[Proof of Corollary~\ref{c:main_corollary}]
Let $WG$ denote the total space of the universal principal twisted cartesian product $WG\to\overline WG$, see e.g.~\cite{May}. Since $G$ is finite, $WG$ is a locally finite simplicial set and thus $C_*WG$ is a locally finite $\bbZ$-complex. The standard results of effective algebraic topology, see e.g.~\cite[Theorem~3.16]{polypost}, provide a strong equivalence of $C_*K(\pi,n)$ with a locally finite $\bbZ$-complex $D$ (one says that $K(\pi,n)$ has effective homology). The Eilenberg-Zilber theorem, or rather its algorithmic version, see e.g.~\cite[Theorem~124]{Sergeraert}, then provides a reduction $M=C_*(WG\times K(\pi,n))\Ra C_*WG\otimes C_*K(\pi,n)$. Composing with the previous, one obtains a strong equivalence $M\LRa C_*WG\otimes D$ with a locally finite $\bbZ$-complex $C_*WG\otimes D$. Theorem~\ref{t:main_theorem} then constructs a $\mcG$-linear strong equivalence of $M$ with a locally finite $\mcG$-complex $N=B(C_*WG\otimes D)$. Thus the (co)homology groups of
\[C_*(WG\times_GK(\pi,n))\cong C_*(WG\times K(\pi,n))/G=M/G,\]
are isomorphic to the (co)homology groups of $N/G$ and these may be computed e.g.~by a simple application of the Smith normal form of the differentials in $N/G$. The Smith normal form can be even computed in polynomial time, see~\cite{SNF}.
\end{proof}

\section{Notes}

\subsection*{A note on homotopy invariance of $\mcR_\infty$-chain maps}
The content of this short note is to prove the following lemma.

\begin{lemma}
Let $g_0\colon M\to N$ be the bottom part of an $\mcR_\infty$-chain map $g_*$ of degree $d$ and let $g_0$ be homotopic to $f_0$. Then, one can extend $f_0$ to an $\mcR_\infty$-chain map $f_*$.
\end{lemma}

\begin{proof}
By the additivity of the equations \eqref{e:G_map}, it is enough to extend any null-homotopic $f_0=[\partial,\eta]$ to an $\mcR_\infty$-chain map $f_*$. We set
\[f_\ell|r_1|\cdots|r_\ell)=\sign{d+1}[\eta,(r_1,\ldots,r_\ell)].\]
Then by the graded Leibniz rule,
\begin{align*}
[\partial,f_\ell]|r_1|\cdots|r_\ell) & =\sign{d+1}[\partial,[\eta,(r_1,\ldots,r_\ell)]]-\sign{d}f_\ell\partial^\otimes|r_1|\cdots|r_\ell) \\
& =\sign{d+1}[[\partial,\eta],(r_1,\ldots,r_\ell)]+[\eta,\partial(r_1,\ldots,r_\ell)]-[\eta,\partial^\otimes(r_1,\ldots,r_\ell)] \\
& =\sign{d+1}[f_0,(r_1,\ldots,r_\ell)]+[\eta,\partial^\mathrm{alg}(r_1,\ldots,r_\ell)].
\end{align*}
The first term equals $\sign{1+d}f_0(r_1,\ldots,r_\ell)+\sign{d(\ell+\degree r1\ell)}(r_1,\ldots,r_\ell)f_0$ and the second term, by the definition and the graded Leibniz rule again, equals
\begin{align*}
& \phantom{{}={}} \sum_{k=1}^{\ell-1}\sign{k-1+\degree r1k}[\eta,(r_1,\ldots,r_kr_{k+1},\ldots,r_\ell)] \\
& \ \ \ +\sum_{k=1}^{\ell-1}\sign{k+\degree r1k}[\eta,(r_1,\ldots,r_k)(r_{k+1},\ldots,r_\ell)] \\
& =\sum_{k=1}^{\ell-1}\sign{k+d+\degree r1k}f_{\ell-1}|r_1|\cdots|r_kr_{k+1}|\cdots|r_\ell) \\
& \ \ \ +\sum_{k=1}^{\ell-1}\sign{k+1+d+\degree r1k}f_k|r_1|\cdots|r_k)(r_{k+1},\ldots,r_\ell) \\
& \ \ \ +\sum_{k=1}^{\ell-1}\sign{d(k+\degree r1k)}(r_1,\ldots,r_k)f_{\ell-k}|r_{k+1}|\cdots|r_\ell).
\end{align*}
Consequently, the $f_\ell$ satisfy the equations \eqref{e:G_map} with $f'_\ell=0$ and thus prescribe an $\mcR_\infty$-chain map of degree $d$.
\end{proof}

\subsection*{A note on dg-categories}
The above works for any locally free dg-category instead of a dga $\mcR$. The definition of an $\mcR_\infty$-map is a more economic version of an $\mcC_\infty$-module for the dg-category
\[\mcC=\mcR\otimes(\xymatrix@1{\bullet\ar[r]^{\bbZ f}&\bullet})\]
describing $\mcR$-linear maps: it consists of two objects with endomorphisms forming $\mcR$ and a map $f$ of degree $d$ between them respecting this action. The corresponding components of an $\mcR_\infty$-map are as follows
\[f_\ell|r_1|\cdots|r_\ell)=\sum_{k=0}^\ell \sign{(d+1)(k+\degree r1k)}(r_1,\ldots,r_k,f,r_{k+1},\ldots,r_\ell).\]

\end{document}